\documentclass[11pt]{article}
\PassOptionsToPackage{obeyspaces}{url}
\usepackage[hidelinks]{hyperref}
\usepackage{bookmark}
\bookmarksetup{open,numbered,depth=5}
\usepackage{amsfonts}
\usepackage{amsmath}
\usepackage{amssymb, amsthm, enumitem}
\usepackage{pgf,tikz}
\usepackage{bm}
\usepackage{cancel}
\usepackage{comment}
\usetikzlibrary{calc}

\usepackage{etoolbox} 
\patchcmd{\thebibliography}{\leftmargin\labelwidth}{\leftmargin\labelwidth\addtolength\itemsep{-0.1\baselineskip}}{}{}

\usepackage{mathtools}
\usepackage{xstring}

\author{Zichao Dong\thanks{Extremal Combinatorics and Probability Group (ECOPRO), Institute for Basic Science (IBS), Daejeon, South Korea. Supported by the Institute for Basic Science (IBS-R029-C4). \texttt{zichao@ibs.re.kr}. } \and Zijian Xu\thanks{School of Mathematical Sciences, Peking University, Beijing 100871. \texttt{xuzijian2005@stu.pku.edu.cn}.}}

\title{Large grid subsets without many cospherical points}
\date{}

\usepackage[nameinlink]{cleveref}

\newtheorem{theorem}{Theorem}

\newtheorem{corollary}[theorem]{Corollary}

\newtheorem{claim}[theorem]{Claim}
\newtheorem{problem}[theorem]{Problem}

\newenvironment{poc}{\begin{proof}[Proof of claim]}{\end{proof}}

\newcommand*{\eqdef}{\stackrel{\mbox{\normalfont\tiny def}}{=}} 

\newcommand*{\R}{\mathbb{R}}                                    
\newcommand*{\F}{\mathbb{F}}

\newcommand{\plane}{\begin{tikzpicture}[scale = 0.15]
    \clip (-0.5, -0.5) rectangle (1.1, 1.5);
    \draw (0, 0) rectangle (1, 1);
\end{tikzpicture}}
\newcommand{\sphere}{\begin{tikzpicture}[scale = 0.15]
    \clip (-0.5, -0.5) rectangle (1.1, 1.5);
    \draw (0.5, 0.5) circle (0.5);
\end{tikzpicture}}

\DeclareMathOperator{\ex}{ex}
\newcommand{\exl}[1]{\ex_{\plane}\bigl(#1\bigr)}
\newcommand{\exc}[1]{\ex_{\sphere}\bigl(#1\bigr)}
\newcommand{\exlc}[1]{\ex\bigl(#1\bigr)}

\allowdisplaybreaks[4]

\begin{document}

\maketitle

\begin{abstract}
	Motivated by intuitions from projective algebraic geometry, we provide a novel construction of subsets of the $d$-dimensional grid $[n]^d$ of size $n - o(n)$ with no $d + 2$ points on a sphere or a hyperplane. For $d = 2$, this improves the previously best known lower bound of $n/4$ toward the Erdős–Purdy problem due to Thiele in 1995. For $d \ge 3$, this improves the recent $\Omega \bigl( n^{\frac{3}{d+1}-o(1)} \bigr)$ bound due to Suk and White, confirming their conjectured $\Omega \bigl( n^{\frac{d}{d+1}} \bigr)$ bound in a strong sense, and asymptotically resolves the generalized Erd\H{o}s--Purdy problem posed by Brass, Moser, and Pach. 
\end{abstract}

\section{Introduction} \label{sec:intro}

In the $n \times n$ grid $[n]^2$, what is the largest size of a subset without $3$ collinear points? This so-called \textsf{no-three-in-line} problem was originally raised by Dudeney \cite{dudeney} in 1917 in the $n = 8$ special case. Since one can pick at most two points on each line, there is an obvious $2n$ upper bound. Many authors (see e.g.~\cite{flammenkamp}) have published constructions towards this problem obtaining the bound $2n$ for $n \le 52$. 

In general, we denote by $\exl{[n]^2; \ell}$ the maximum size of a subset of $[n]^2$ which contains no $\ell$ collinear points. Then an obvious upper bound is that $\exl{[n]^2; k+1} \le kn$, and we have seen that $\exl{[n]^2, 3} = 2n$ for $n = 1, \dots, 52$. The best-known lower bound for the \textsf{no-three-in-line} problem is $\exl{[n]^2, 3} \ge \frac{3}{2}n - o(n)$ due to Hall et al.~\cite{hall_jackson_sudbery_wild}. This problem is both fascinating and enigmatic, and there remains a significant lack of consensus on what the answer is supposed to be: 
\vspace{-0.5em}
\begin{itemize}
    \item Green \cite[Problem 72]{green} believes that the answer should be approximately $\frac{3}{2}n$. 
    \vspace{-0.5em}
    \item Brass, Moser, and Pach \cite{brass_moser_pach} think that roughly $2n$ points could always be achievable. 
    \vspace{-0.5em}
    \item Guy and Kelly \cite{guy_kelly} conjecture a $1.87n$ answer based on some probabilistic evidence. 
\end{itemize}
\vspace{-0.5em}
Very recently, Kov\'{a}cs, Nagy, and Szab\'{o} \cite{kovacs_nagy_szabo} proved $\exl{[n]^2; k+1} = kn$ for all $k > 15\sqrt{n\log n}$. 

Erd\H{o}s and Purdy proposed a similar \textsf{no-four-on-a-circle} problem in 1981 (see \cite{guy}), asking for the largest size of a subset of $[n]^2$ containing no $4$ concyclic points. Here, lines are considered as circles of infinite radius and thus are included in this definition. Since every horizontal line intersects such a subset in at most $3$ points, we have a trivial upper bound $3n$. Thiele \cite{thiele_paper,thiele_thesis} improved this upper bound to $2.5n$, and he also provided a $0.25n$ lower bound via algebraic constructions. 

In this paper, we study the high dimension generalization of the \textsf{no-four-on-a-circle} problem: 

\begin{problem}[Brass--Moser--Pach {\cite[Problem 4 of Chapter 10.1]{brass_moser_pach}}] \label{prob:BMP}
    What is the maximum number of points that can be selected from the lattice cube $[n]^d$ with no $d + 2$ points on a plane or a sphere? 
\end{problem}

To clarify the terminologies \emph{sphere} and \emph{plane}, in Euclidean space $\R^d$ we refer to 
\vspace{-0.5em}
\begin{itemize}
    \item a \emph{plane} as a $(d-1)$-dimensional hyperplane, and a \emph{sphere} as a $(d-1)$-dimensional sphere. 
\end{itemize}
\vspace{-0.5em}
In this language, coplanar (resp.~cospherical) points are points lying on a fixed plane (resp.~sphere). Let $\exc{[n]^d; \ell}$ be the maximum size of a subset of $[n]^d$ with no $\ell$ cospherical points nor $\ell$ coplanar points. More generally, denote by $\exlc{[n]^d; p_{\plane}, s_{\sphere}}$ the maximum size of a subset of $[n]^d$ containing neither $p$ coplanar points nor $s$ cospherical points. Under the notations above, 
\vspace{-0.5em}
\begin{itemize}
    \item The \textsf{no-three-in-line} problem asks to determine $\exl{[n]^2; 3} = \exlc{[n]^2; 3_{\plane}, \infty_{\sphere}}$, and
    \vspace{-0.5em}
    \item the \textsf{no-four-on-a-circle} problem asks to determine $\exc{[n]^2; 4} = \exlc{[n]^2; 4_{\plane}, 4_{\sphere}}$. 
\end{itemize}
\vspace{-0.5em}
Thiele \cite{thiele_thesis} deduced the following bounds concerning the Erd\H{o}s--Purdy \textsf{no-four-on-a-circle} problem: 
\[
\frac{n}{4} < \exlc{[n]^2; 3_{\plane}, 4_{\sphere}} \le \exc{[n]^2; 4} \le \exlc{[n]^2; \infty_{\plane}, 4_{\sphere}} \le \frac{5n-3}{2}. 
\]
In higher dimensions, by partitioning the grid into horizontal planes one can deduce an easy linear upper bound $\exc{[n]^d; d+2} \le (d+1)n$. The lower bound direction seems to be essentially harder. Thiele \cite{thiele_thesis} established $\exc{[n]^d; d+2} \ge \Omega \bigl( n^{\frac{1}{d-1}} \bigr)$. Very recently, Suk and White \cite{suk_white} obtained an improved bound $\exc{[n]^d; d+2} \ge \Omega \bigl( n^{\frac{3}{d+1}-o(1)} \bigr)$, and conjectured that $\exc{[n]^d; d+2} \ge \Omega \bigl( n^{\frac{d}{d+1}} \bigr)$. 

\smallskip

Our main result shows that for each $d$, the function $\exc{[n]^d; d+2}$ always grows linearly in $n$. 

\begin{theorem} \label{thm:d+2_d+2}
    For every integer $d \ge 2$, we have $\exc{[n]^d; d+2} \ge n - o(n)$. 
\end{theorem}

\Cref{thm:d+2_d+2} resolves the Suk--White conjecture \cite[Conjecture 1.2]{suk_white} in a strong form. Moreover, it provides an answer to \Cref{prob:BMP} which is asymptotically tight in $n$. It is also worth mentioning that \Cref{thm:d+2_d+2} significantly improves upon Thiele's lower bound on $\exc{[n^2]; 4}$ from 1995. 

Recall that Thiele \cite{thiele_thesis} indeed deduced that $\exlc{[n]^2; 3_{\plane}, 4_{\sphere}} > n/4$, implying $\exc{[n]^2; 4} > n/4$. With some extra work, we can also generalize this stronger result to higher dimensions. 

\begin{theorem} \label{thm:d+1_d+2}
    For every integer $d \ge 2$, we have $\exlc{[n]^d; (d+1)_{\plane}, (d+2)_{\sphere}} \ge \frac{n}{d+1} - o(n)$. 
\end{theorem}

\paragraph{Proof overview and organization.} Thiele \cite{thiele_thesis} deduced $\exc{[n]^d; d+2} \ge \Omega \bigl( n^{\frac{1}{d-1}} \bigr)$ bounds via moment curves. Specifically, his constructions were based on $\bigl\{ (t, t^2, \dots, t^d) \bigr\} \subset \F_p^d$. People observed further \cite{roth,brass_knauer,lefmann,suk_white} that such constructions could give $\exlc{[n]^d, d+1, 2d} \ge \Omega(n)$ bounds. 

One might try to come up with other curves to achieve better lower bounds. Bezout's theorem from algebraic geometry shows us that a curve of degree $k$ intersects a typical sphere in $2k$ points (including projective infinite points). To avoid $d + 2$ cospherical points, the curve is supposed to be of degree at most $\frac{d+1}{2}$. However, when $d \ge 3$, this is too small for a curve to escape a single plane. 

To address this caveat, our strategy is to utilize infinite points. Projectively, every sphere passes through a same set of infinite points. If our selected curve goes through many of these infinite points as well, then the number of spherical intersections decreases drastically. This motivates us to analyze rational curves $\Bigl\{\bigl( \frac{f_1(t)}{h(t)}, \, \dots, \, \frac{f_d(t)}{h(t)} \bigr)\Bigr\} \subset \F_p^d$, where $f_1, \dots, f_d, h$ are well-behaved polynomials. 


\smallskip

We shall find desired polynomials in \Cref{sec:poly}, and prove \Cref{thm:d+2_d+2,thm:d+1_d+2} in \Cref{sec:proof}. 

\section{Constructing the polynomials} \label{sec:poly}

Throughout this section, we take $p$ as a fixed prime number with $p \equiv 1 \pmod 4$ and $p > (d+1)!$. Importantly, the square root of $-1$ exists in the $p$-element field $\F_p$. Take $\alpha \in \F_p$ with $\alpha^2 + 1 = 0$. 

In the univariate polynomial ring $\F_p[t]$, we call a polynomial $f$ \emph{nice} if the linear coefficient of $f$ vanishes. That is, a nice $f$ takes the form $a_0 + a_2 t^2 + a_3 t^3 + \cdots$, where $a_0, a_2, a_3, \dots \in \F_p$. 

\begin{theorem} \label{thm:poly_nice}
    Suppose $d \ge 2$ is an integer. Let $p$ be a prime with $p \equiv 1 \pmod 4$ and $p > (d+1)!$. Then there exist $d + 2$ polynomials $f_1, \dots, f_d, g, h \in \F_p[t]$ with the following properties. 
    \vspace{-0.5em}
    \begin{itemize}
        \item \emph{(\textsf{identity})} They satisfy the polynomial identity $f_1^2 + \dots + f_d^2 = gh$ in $\F_p[t]$. 
        \vspace{-0.5em}
        \item \emph{(\textsf{degree})} They have $(\deg f_1, \dots, \deg f_d, \deg g, \deg h) = (d, \dots, d, d - 1, d + 1)$. 
        \vspace{-0.5em}
        \item \emph{(\textsf{independence})} The $d+2$ polynomials $f_1, \dots, f_d, g, h$ are linearly independent. 
        \vspace{-0.5em}
        \item \emph{(\textsf{vanish})} The $d+1$ polynomials $f_1, \dots, f_d, h$ are nice. (Note that $g$ is excluded.) 
    \end{itemize}
    \vspace{-0.5em}
\end{theorem}

\noindent \textbf{Remark.} We shall make use of \textsf{vanish} only in the proof of \Cref{thm:d+1_d+2}. 

\begin{proof}
    We first construct $f_1, \dots, f_d, g, h$ satisfying \textsf{identity}, \textsf{degree}, and \textsf{independence}. 
    
    Suppose $\lambda_1, \dots, \lambda_d \in \F_p$ are distinct and set $g_j \eqdef (t - \lambda_1) \cdots \widehat{(t - \lambda_j)} \cdots (t - \lambda_d)$ for $j = 1, \dots, d$. Observe that $g_1, \dots, g_d$ are linearly independent polynomials of degree $d - 1$. 
    
    Let $A = (a_{i, j})$ be an $d \times d$ matrix. For $i = 1, \dots, d$, we define degree-$d$ polynomials
    \[
    f_i = \sum_{j=1}^d \frac{a_{i, j}}{g_j(\lambda_j)} \cdot g_j + \prod_{j=1}^d (t - \lambda_j). 
    \]
    Observe that $f_i(\lambda_j) = a_{i, j}$ holds for every $i \in [d]$ and every $j \in [d]$. 

    \begin{claim} \label{claim:g_indep}
        There exists $j \in [d]$ such that $g_j$ does not appear in the linear span of $f_1, \dots, f_d$. 
    \end{claim}

    \begin{poc}
        Suppose to the contrary that each of $g_1, \dots, g_d$ can be written as a linear combination of $f_1, \dots, f_d$. We known that $g_1, \dots, g_d$ form a basis of polynomials in $\F_p[t]$ of degree at most $d - 1$. So do $f_1, \dots, f_d$. However, each of $f_1, \dots, f_d$ is of degree $d$, a contradiction. 
    \end{poc}

    \begin{claim} \label{claim:f_indep}
        If the matrix $A$ is invertible, then the polynomials $f_1, \dots, f_d$ are linearly independent. 
    \end{claim}

    \begin{poc}
        Assume for the sake of contradiction that there exist coefficients $c_1, \dots, c_d$ that are not all zero with $c_1 f_1 + \dots + c_d f_d = 0$. Evaluating at $\lambda_1, \dots, \lambda_d$, it follows that $c_1 a_{1, j} + \dots + c_d a_{d, j}$ holds for every $j \in [d]$. This implies that the row vectors of $A$ are linearly dependent with respect to coefficients $c_1, \dots, c_d$, contradicting the assumption that $A$ is invertible. 
    \end{poc}

    Set $\lambda_i \eqdef i - 1$ in $\F_p$ for $i = 1, \dots, d$. Take $\alpha \in \F_p$ with $\alpha^2 + 1 = 0$ and specify the matrix $A$ as
    \[
    a_{i, j} \eqdef \begin{cases}
        1 \qquad &\text{if $i = j$}, \\[-0.5em]
        \alpha \qquad &\text{if $i = j + 1$}, \\[-0.5em]
        \pm \alpha \qquad &\text{if $(i, j) = (1, d)$}, \\[-0.5em]
        0 \qquad &\text{otherwise}. 
    \end{cases}
    \]
    Here the ``$a_{1, d} = \pm \alpha$'' entry is chosen so that $A$ is invertible. In fact, it is not hard to compute that $A$ has determinant $1 \pm \alpha^d$. Since $1 + \alpha^d = 1 - \alpha^d = 0$ cannot happen in $\F_p$, we may always find a legal choice of $a_{1, d}$. As an example, we write out the matrix $A$ below in the $d = 6$ case: 
    \[
    A = \begin{pmatrix}
        1 &\!\! 0 &\!\! 0 &\!\! 0 &\!\! 0 &\!\! \alpha \\[-0.5em]
        \alpha &\!\! 1 &\!\! 0 &\!\! 0 &\!\! 0 &\!\! 0 \\[-0.5em]
        0 &\!\! \alpha &\!\! 1 &\!\! 0 &\!\! 0 &\!\! 0 \\[-0.5em]
        0 &\!\! 0 &\!\! \alpha &\!\! 1 &\!\! 0 &\!\! 0 \\[-0.5em]
        0 &\!\! 0 &\!\! 0 &\!\! \alpha &\!\! 1 &\!\! 0 \\[-0.5em]
        0 &\!\! 0 &\!\! 0 &\!\! 0 &\!\! \alpha &\!\! 1
    \end{pmatrix}. 
    \]
    It is important that $\lambda_1 = 0$ and the sum of squares of entries in each column of $A$ equals $0$. 
    
    \begin{claim} \label{claim:division}
        The polynomial $(t-\lambda_1) \cdots (t-\lambda_d)$ divides $f_1^2 + \cdots + f_d^2$. 
    \end{claim}
    
    \begin{poc}
        This is true because $f_1^2(\lambda_k) + \dots + f_d^2(\lambda_k) = 0$ holds for all $k \in [d]$. 
    \end{poc}
    
    Pick $g \eqdef g_j$ as in \Cref{claim:g_indep}. Then \Cref{claim:division} implies that $g$ divides $f_1^2 + \dots + f_d^2$, and we set $h$ as the quotient, implying \textsf{identity}. Since the leading coefficient of $f_1^2 + \dots + f_d^2$ is $d \ne 0$ in $\F_p$, we have 
    \[
    \deg(f_1^2 + \dots + f_d^2) = 2d, \qquad \deg(g) = d-1, 
    \]
    and so $\deg(h) = d+1$, implying \textsf{degree}. Since $h$ is is of higher degree than $f_1, \dots, f_d, g$, evidently $h$ is linearly independent on $f_1, \dots, f_d, g$. So, \textsf{independence} follows from \Cref{claim:f_indep} and \Cref{claim:g_indep}. 

    \smallskip
    We then fix $g$ and tweak $f_1, \dots, f_d$ to fulfill \textsf{vanish}. We sketch the strategy below: 
    \vspace{-0.5em}
    \begin{itemize}
        \item Replace $f_i$ by $f_i + \mu_i g$, where $\mu_i$ is appropriately chosen so that every $f_i + \mu_i g$ is nice. 
        \vspace{-0.5em}
        \item Replace $f_1$ by $f_1 + \nu x g$, where $\nu$ is appropriately chosen so that the resulted $h$ is nice. 
    \end{itemize}
    \vspace{-0.5em}
    To achieve this, the crucial properties of $g$ we require are that $g$ is not nice and that $g(0) = 0$. From the assumption $p > (d+1)!$ we deduce that each of $g_1, \dots, g_d$ is not nice, and so $g$ is not nice. To further enforce $g(0) = 0$, we need the following stronger version of \Cref{claim:g_indep} under our specific $A$. 

    \begin{claim} \label{g_indep+}
        There exists $j \in [d] \setminus \{1\}$ such that $g_j$ does not appear in the linear span of $f_1, \dots, f_d$. 
    \end{claim}

    \begin{poc}
        Suppose to the contrary that each of $g_1, \dots, g_d$ can be written as a linear combination of $f_1, \dots, f_d$. Our choice of $A$ implies $f_1 = \frac{1}{g_1(\lambda_1)} \cdot g_1 + \frac{\alpha}{g_2(\lambda_2)} \cdot g_2$, and so $g_1$ can also be written as a linear combination of $f_1, \dots, f_d$. However, this contradicts \Cref{claim:g_indep}. 
    \end{poc}

    Our choice of $\lambda_1, \dots, \lambda_d$ implies that $g_2(0) = \dots = g_d(0) = 0$ in $\F_p$. Thanks to \Cref{claim:g_indep}, we can choose $g$ to be one of $g_2, \dots, g_d$ so that $g(0) = 0$ without spoiling \textsf{identity}, \textsf{degree}, and \textsf{independence}. 
    
    Since $g$ is not nice, the linear coefficient of $g$ is nonzero in $\F_p$. It follows that for each $i = 1, \dots, d$, there exists a unique $\mu_i \in \F_p$ such that the linear coefficient of $f_i + \mu_i g$ vanishes. Set 
    \[
    \widetilde{f}_1 \eqdef f_1 + (\mu_1 + \nu t)g, \quad \widetilde{f}_2 \eqdef f_2 + \mu_2 g, \quad \dots, \quad \widetilde{f}_d \eqdef f_d + \mu_d g, \quad \widetilde{h} \eqdef \frac{\widetilde{f}_1^2 + \dots + \widetilde{f}_d^2}{g}. 
    \]
    Then $\widetilde{f}_1, \dots, \widetilde{f}_d$ are nice while \textsf{identity}, \textsf{degree}, and \textsf{independence} are left intact. It remains to verify that there is an appropriate choice of $\nu$ to make $\widetilde{h}$ nice. We compute
    \begin{equation} \label{eq:h_tilde}
    \widetilde{h} = h + \sum_{i=1}^d \mu_i (2f_i + \mu_i g) + 2(f_1 + \mu_1 g)\nu t + \nu^2 t^2 g. 
    \end{equation}
    Denote by $[t]f$ the linear coefficient of polynomial $f$. Since $g(0) = 0$ and 
    \[
    g_2(0) = 0 \implies f_1(0) = \frac{1}{g_1(\lambda_1)} \cdot g_1(0) + \frac{\alpha}{g_2(\lambda_2)} \cdot g_2(0) = \frac{1}{g_1(0)} \cdot g_1(0) + \frac{\alpha}{g_2(1)} \cdot g_2(0) = 1, 
    \]
    from \eqref{eq:h_tilde} we deduce that 
    \[
    [t] \widetilde{h} = [t] \biggl( h + \sum_{i=1}^d \mu_i (2f_i + \mu_i g) \biggr) + 2 \bigl( f_1(0) + \mu_1 g(0) \bigr) \nu = 2\nu + [t] \biggl( h + \sum_{i=1}^d \mu_i (2f_i + \mu_i g) \biggr). 
    \]
    It follows that $[t]\widetilde{h}$ is mapped surjectively onto $\F_p$ as $\nu$ varies. In particular, there is a unique choice of $\nu$ such that $[t]\widetilde{h} = 0$, implying \textsf{vanish}. We thus conclude the proof of \Cref{thm:poly_nice}. 
\end{proof}

\section{Proofs of \texorpdfstring{\Cref{thm:d+2_d+2,thm:d+1_d+2}}{Theorems 1 and 2}} \label{sec:proof}

We need the prime number distribution theorem for arithmetic progressions (see e.g.~\cite{apostol,page}). 

\begin{theorem}
    Let $\pi(N; k, \ell)$ be the number of primes $p$ with $p \le N$ and $p \equiv k \pmod {\ell}$. Then
    \[
    \biggl| \pi(N; k, \ell) - \frac{1}{\phi(\ell)} \int_2^x \frac{dt}{\log t} \biggr| = O \biggl( \frac{N}{\exp \bigl( c\sqrt{\log N}\bigr)} \biggr) \qquad \text{as $N \to \infty$}, 
    \]
    where $\phi$ is the Euler totient function and $c$ is an absolute positive constant. 
\end{theorem}

\begin{corollary} \label{coro:prime}
    As $n \to \infty$, the smallest prime $p \ge n$ with $p \equiv 1 \pmod 4$ satisfies
    \[
    p = n + O \biggl( \frac{n \log n}{\exp \bigl( c\sqrt{\log n}\bigr)} \biggr) = n + o(n). 
    \]
\end{corollary}

\begin{proof}[Proof of \Cref{thm:d+2_d+2}]
    The strategy is to construct for every prime $p \equiv 1 \pmod 4$ a subset $S \subset \F_p^d$ of size $p - d - 2$ containing neither $d + 2$ points being coplanar nor $d + 2$ points being cospherical. To be specific, we shall identify $[n]^d$ as $\{\overline{1}, \overline{2}, \dots, \overline{n}\}^d$ in $\F_p^d$. Starting with a claimed $S$ and considering random translations, a direct first moment calculation shows that there exists $\mathbf{v} \in \F_p^d$ with 
    \[
    \bigl| (S + \mathbf{v}) \cap [n]^d \bigr| \ge \Bigl( \frac{n}{p} \Bigr)^d \cdot |S|. 
    \]
    Due to \Cref{coro:prime}, if we take the smallest prime $p$ with $p > \max \bigl\{ 10(d+1)!, n \bigr\}$ and $p \equiv 1 \pmod 4$, then one can compute that $\bigl| (S + \mathbf{v}) \cap [n]^d \bigr| = n - o(n)$. Since every plane (resp.~sphere) in $[n]^d$ is a plane (resp.~sphere) in $\F_p^d$, the set $(S + \mathbf{v}) \cap [n]^d$ concludes the proof of \Cref{thm:d+2_d+2}. 
    
    Pick polynomials $f_1, \dots, f_d, g, h$ as in \Cref{thm:poly_nice}. Consider the parametric curve 
    \[
    \gamma \colon \bigl\{ t \in \F_p : h(t) \ne 0 \bigr\} \to \F_p^d, \qquad t \mapsto \biggl( \frac{f_1(t)}{h(t)}, \, \frac{f_2(t)}{h(t)}, \, \dots, \, \frac{f_d(t)}{h(t)} \biggr). 
    \]
    We are going to show that the image of $\gamma$ provides a valid construction of $S$. 
    
    Let $\pi \eqdef \{\lambda_1 x_1 + \dots + \lambda_d x_d + \lambda = 0\} \subset \F_p^d$ be an arbitrary plane. If $\gamma(t_0) \in \pi$, then $t_0$ is a zero of the polynomial $P \eqdef \lambda_1 f_1 + \dots + \lambda_d f_d + \lambda h$. It follows from \textsf{degree} and \textsf{independence} in \Cref{thm:poly_nice} that $\deg P \le d + 1$ and $P \ne 0$. Therefore, the number of $t_0$ is upper bounded by $d + 1$. 

    Let $\omega \eqdef \{x_1^2 + \dots + x_d^2 + \mu_1 x_1 + \dots + \mu_d x_d + \mu = 0\} \subset \F_p^d$ be an arbitrary sphere. If $\gamma(t_0) \in \omega$, then \textsf{identity} in \Cref{thm:poly_nice} implies that $t_0$ is a zero of $Q \eqdef g + \mu_1 x_1 + \dots + \mu_d x_d + \mu$. Since \textsf{degree} and \textsf{independence} show that $\deg Q \le d + 1$ and $Q \ne 0$, the number of $t_0$ is upper bounded by $d + 1$. 
    
    We have shown that the image of $\gamma$ contains neither $d + 2$ coplanar points nor $d + 2$ cospherical points in $\F_p$, and we are left to prove that this image has at least $p - d - 2$ distinct points. To see this, noticing $\deg h = d + 1$ and $\bigl| \bigl\{ t \in \F_p : h(t) = 0 \bigr\} \bigr| \le d + 1$, it suffices to show that $\gamma$ has at most one \emph{self-intersection}, which refers to a pair of distinct $t_1, t_2$ with $\gamma(t_1) = \gamma(t_2)$. If this is not the case, then we can find a $(d+2)$-element $T \subset \F_d$ with $h(t) \ne 0 \, (\forall t \in T)$ and $|\gamma(T)| = \bigl| \bigl\{ \gamma(t) : t \in T \bigr\} \bigr| \le d$. So, $\gamma(T)$ is contained in some plane $\{\rho_1 x_1 + \dots + \rho_d x_d + \rho = 0\}$, hence $R \eqdef \rho_1 f_1 + \dots + \rho_d f_d + \rho f$ has $d + 2$ zeros in $\F_p$. However, \textsf{degree} and \textsf{independence} in \Cref{thm:poly_nice} tell us that $\deg R \le d + 1$ and $R \ne 0$, a contradiction. Therefore, the proof of \Cref{thm:d+2_d+2} is complete. 
\end{proof}

\begin{proof}[Proof of \Cref{thm:d+1_d+2}]
    Similar to the proof of \Cref{thm:d+2_d+2}, by embedding $[n]^d$ into $\F_p^d$, \Cref{coro:prime} tells us that it suffices to construct for every prime $p \equiv 1 \pmod 4$ a subset $\widetilde{S} \subset \F_p^d$ of size $\frac{1 - o(1)}{d+1} \cdot p$. 

    Pick polynomials $f_1, \dots, f_d, g, h$ as in \Cref{thm:poly_nice}. Consider the parametric curve 
    \[
    \widetilde{\gamma} \colon \biggl\{ t \in \F_p : t^{-1} \in \Bigl\{ \overline{1}, \overline{2}, \dots, \overline{\bigl\lfloor \tfrac{p-1}{d+1} \bigr\rfloor} \Bigr\}, \, h(t) \ne 0 \biggr\} \to \F_p^d \qquad t \mapsto \biggl( \frac{f_1(t)}{h(t)}, \, \frac{f_2(t)}{h(t)}, \, \dots, \, \frac{f_d(t)}{h(t)} \biggr). 
    \]
    We are going to show that the image of $\gamma$ provides a valid construction of $S$. The same arguments as in the proof of \Cref{thm:d+2_d+2} show that this image, containing at least $\bigl\lfloor \frac{p-1}{d+1} \bigr\rfloor -  d - 2$ distinct points, contains no $d + 2$ cospherical points. We are left to show that no plane intersects it in $d + 1$ points. 
    
    For any $\pi \eqdef \{\lambda_1 x_1 + \dots + \lambda_d x_d + \lambda = 0\} \subset \F_p^d$, if $\widetilde{\gamma}(t_0) \in \pi$, then $t_0$ is a zero of the polynomial $P \eqdef \lambda_1 f_1 + \dots + \lambda_d f_d + \lambda h = 0$, where \textsf{degree}, \textsf{independence} in \Cref{thm:poly_nice}
    show that $\deg P \le d + 1$ and $P \ne 0$. Moreover, \textsf{nice} implies that its linear coefficient vanishes (i.e., $[x]P = 0$). If $P$ has $d + 1$ distinct zeros $t_1, \dots, t_{d+1}$ in $\{\overline{1}, \overline{2}, \dots, \overline{n}\}$, then $[x]P = 0$ and Vieta's theorem tell us that 
    \[
    t_1^{-1} + \dots + t_{d+1}^{-1} = 0. 
    \]
    This implies that such zeros cannot live in the domain of $\widetilde{\gamma}$ simultaneously. Thus, each plane in $\F_p^d$ intersects $\widetilde{\gamma}$ in at most $d$ points. The proof of \Cref{thm:d+1_d+2} is complete. 
\end{proof}

\section*{Acknowledgments}

The authors would like to thank Hong Liu for comments and suggestions to an earlier draft of this manuscript. Part of this work was done during a visit of the first author to Peking University. He is thankful to Chunwei Song for hosting. 

\bibliographystyle{plain}
\bibliography{plane_sphere}

\end{document}